\documentclass[12pt,reqno]{amsart}
\usepackage[utf8]{inputenc}
\usepackage{amsopn,amssymb,mathrsfs,mathrsfs,a4wide}

\newtheorem{thm}{Theorem}[section]
\newtheorem{prop}[thm]{Proposition}
\newtheorem{cor}[thm]{Corollary}
\newtheorem{fact}[thm]{Fact}
\newtheorem{prob}[thm]{Problem}

\theoremstyle{definition}
\newtheorem{defn}[thm]{Definition}
\newtheorem{example}[thm]{Example}
\newtheorem{remark}[thm]{Remark}

\newcommand{\cl}[1]{\ensuremath{\overline{{#1}}}}

\newcommand{\map}[3]{\ensuremath{{#1}:{#2}\to{#3}}}
\newcommand{\N}{\mathbb{N}}
\newcommand{\n}[1]{\ensuremath{\left\|{#1}\right\|}}
\newcommand{\ndot}{\ensuremath{\left\|\cdot\right\|}}
\newcommand{\pn}[2]{\ensuremath{\left\|{#1}\right\|_{#2}}}
\newcommand{\pndot}[1]{\ensuremath{\left\|\cdot\right\|_{#1}}}
\newcommand{\set}[2]{\ensuremath{\left\{{#1}\;:\;\,{#2}\right\}}}
\newcommand{\ts}{\textstyle}
\newcommand{\tn}[1]{\ensuremath{\tri{#1}\tri}}
\newcommand{\tndot}{\ensuremath{\tri\cdot\tri}}
\newcommand{\tri}{{\displaystyle |\kern-.9pt|\kern-.9pt|}}

\DeclareMathOperator{\supp}{supp}
\DeclareMathOperator{\lspan}{span}

\newcommand{\st}{($*$)}

\title{Polyhedrality and decomposition}

\author[T.~A.~Abrahamsen]{Trond A.~Abrahamsen}
\address{Department of Mathematics, University of Agder, Postboks 422,
  4604 Kristiansand, Norway}
\email{trond.a.abrahamsen@uia.no}
\urladdr{http://home.uia.no/trondaa/index.php3}

\author[V.~P.~Fonf]{Vladimir P.~Fonf}
\address{Department of Mathematics, Ben-Gurion University of the
  Negev, 84105 Beer-Sheva, Israel}
\email{fonf@math.bgu.ac.il}

\author[R.~J.~Smith]{Richard J.~Smith}
\address{School of Mathematics and Statistics, University College
  Dublin, Belfield, Dublin 4, Ireland} \email{richard.smith@maths.ucd.ie}
\urladdr{http://mathsci.ucd.ie/~rsmith}

\author[S.~Troyanski]{Stanimir Troyanski}
\address{Institute of Mathematics and Informatics, Bulgarian Academy
  of Science, bl.8, acad. G.~Bonchev str.~1113 Sofia, Bulgaria and
  Departamento de Matem\'aticas, Universidad de Murcia, Campus de
  Espinardo, 30100 Espinardo (Murcia), Spain}
\email{stroya@um.es}

\thanks{The second author was financially supported by GACR 16-073785
  and RVO:~67985840. The fifth author was partially supported by
  MTM2014-54182-P (MINECO/FEDER), MTM2017-86182-P (AEI/FEDER, UE) and
  the Bulgarian National Scientific Fund under Grant DFNI-I02/10.}
  

\begin{document}

\begin{abstract}
The aim of this note is to present two results that make the task of
finding equivalent polyhedral norms on certain Banach spaces, having
either a Schauder basis or an uncountable unconditional basis, easier
and more transparent. The hypotheses of both results are based on
decomposing the unit sphere of a Banach space into countably many
pieces, such that each one satisfies certain properties. Some examples
of spaces having equivalent polyhedral norms are given.
\end{abstract}

\maketitle

\section{Introduction}\label{sect_introduction}

The concepts of upper and lower $p$-estimates (for disjoint elements)
in Banach lattices, where $1 < p < \infty$, play an important role
when studying the geometry of Banach spaces. More precisely, using
their relationship with $p$-convexity and concavity, it is possible to
find asymptotically sharp estimates at $0$ of the moduli of convexity
and smoothness, and the cotype and type of the Banach lattice (see
e.g.~\cite[Chapter 1]{lt:79}). We introduce an analogue of upper
$p$-estimate in the case $p = \infty$, and in doing so we find
sufficient conditions for isomorphic polyhedral renorming. In our
opinion, these conditions are easier to verify in many concrete
cases. Let us recall that, following V.~Klee \cite{k:59}, a Banach
space is said to be \emph{polyhedral} when the unit balls of all of
its finite-dimensional subspaces are polytopes. A Banach space $X$ is
said to be \emph{isomorphically polyhedral} if it is isomorphic to a
polyhedral space or, equivalently, if $X$ admits an equivalent
polyhedral norm.

We denote by $B_X$ and $S_X$ the (closed) unit ball and unit sphere of
$X$, respectively. Let $X$ have an unconditional basis
$(e_\gamma)_{\gamma \in \Gamma}$, with corresponding biorthogonal
functionals $(e_\gamma^*)_{\gamma\in\Gamma}$. Given a subset $A
\subseteq \Gamma$, we define the projections
\[
P_A x \;=\; \sum_{\gamma \in A} e^*_\gamma(x)e_\gamma
\qquad\text{and}\qquad R_A x \;=\; x-P_A x.
\]
If $(e_j)_{j=1}^\infty$ is a Schauder basis (with corresponding
biorthogonals $(e_j^*)_{j=1}^\infty$), define $P_n=P_{\{1,\dots,n\}}$
and $R_n=R_{\{1,\dots,n\}}$. From time to time we will require the
\emph{support} of an element in $X$ or its dual, with respect to the
given basis:~define
\[
\supp(x) \;=\; \set{\gamma \in \Gamma}{e^*_\gamma(x) \neq 0},
\]
for all $x \in X$ and, given $f \in X^*$, set
\[
\supp(f) \;=\; \set{\gamma \in \Gamma}{f(e_\gamma) \neq 0}.
\]
We will also require a type of function known in approximation theory
as a \emph{modulus}, namely a non-decreasing continuous function
$\map{\omega}{[0,\infty)}{[0,\infty)}$ such that $\omega(0)=0$. We
present our chief definition.

\begin{defn}\label{defn_star} We say that the Banach space $X$ has
  decomposition {\st} (with respect to the unconditional basis
  $(e_\gamma)_{\gamma \in \Gamma}$ and modulus $\omega$) if, for every
  $x \in X$ there exist positive numbers $c(x)$ and $d(x)$, such that
  the inequality
\[
\n{x} \;\leqslant\; \n{P_A x} + c(x)\omega(d(x)\pn{R_A x}{\infty}),\tag*{\st}
\]
holds for every subset $A \subseteq \Gamma$. Here, $\pndot{\infty}$
denotes the \emph{supremum norm} on $X$, i.e.
\[
\pn{x}{\infty} \;=\; \max\set{|e_\gamma^*(x)|}{\gamma \in \Gamma}.
\]
\end{defn}

\begin{remark}\label{rem_star} It is enough that {\st} holds only for
  all $x \in S_X$. Given $x \neq 0$, we can set
\[
c(x) \;=\; \n{x}\cdot c\left(\frac{x}{\n{x}}\right)
\qquad\text{and}\qquad d(x) \;=\; \frac{1}{\n{x}}\cdot
d\left(\frac{x}{\n{x}} \right).
\]
Clearly {\st} holds for $x$ if it holds for $x/\n{x}$.
\end{remark}

Now we present our two main results.

\begin{thm}\label{thm_symmetric} Let a Banach space $X$ have {\st}
  with respect to a symmetric basis $(e_\gamma)_{\gamma \in
    \Gamma}$. Then $X$ admits an equivalent polyhedral norm.
\end{thm}

The proof of this theorem follows from the next result.

\begin{thm}\label{thm_lim_inf} Let $X$ be a Banach space having an
  unconditional basis $(e_\gamma)_{\gamma \in \Gamma}$. Let
  $(a_n)_{n=1}^\infty$ be a sequence of positive numbers tending to
  $0$, such that
\begin{equation}\label{eqn_unc_est}
\lim\inf_{n \to \infty} a_n^{-1}\left(\n{x}-\sup_{|A| \leqslant
    n}\n{P_A x}\right) \;<\; \infty \qquad\text{for every }x \in X.
\end{equation}
Then $X$ admits an equivalent polyhedral norm. 

Alternatively, if $X$ admits a Schauder basis $(e_j)_{j=1}^\infty$, we
can reach the same conclusion if we replace condition
(\ref{eqn_unc_est}) by
\begin{equation}\label{eqn_Sch_est}
\lim\inf_{n \to \infty} a_n^{-1}\left(\n{x}-\n{P_n x}\right) \;<\;
\infty \qquad\text{for every }x \in X.
\end{equation}
\end{thm}

The following remark will be used a few times in proofs throughout the paper.
It also allows us to simplify the expression $\sup_{|A|\leqslant n} \n{P_A x}$
in condition (\ref{eqn_unc_est}), in the event that the basis of $X$ is
$1$-symmetric. 

\begin{remark}\label{rem_A_n(x)}
Given non-zero $x \in X$, where $X$ has an unconditional basis
$(e_\gamma)_{\gamma \in \Gamma}$, we can enumerate $\supp(x)$ as
a (finite or infinite) sequence $(\gamma_k)_{k\geqslant 1}$ of distinct points
in $\Gamma$, in such a way that $|e_{\gamma_1}^*(x)| \geqslant |e_{\gamma_2}^*(x)|
\geqslant |e_{\gamma_3}^*(x)| \dots$. Set $A_n(x)=\{\gamma_1,\dots,
\gamma_n\}$ (or set $A_n(x)=\supp(x)$ if $|\supp(x)| < n$). If the basis
of $X$ is $1$-symmetric then $\sup_{|A|\leqslant n} \n{P_A x}$ in condition
(\ref{eqn_unc_est}) is equal to $\n{P_{A_n(x)}x}$. The choice of the
$\gamma_k$, and thus the sets $A_n(x)$, may not be unique, however, said
choice will not matter whenever we make use of these sets. 
\end{remark}

Section \ref{sect_examples} is devoted to examples. In it, we present
a series of examples of Banach spaces having {\st}, an example that
exposes the difference between conditions (\ref{eqn_unc_est}) and
(\ref{eqn_Sch_est}) in Theorem \ref{thm_lim_inf}, and an example of a
non-symmetric equivalent norm on $c_0$ that does not satisfy condition
(\ref{eqn_unc_est}) with respect to the usual basis. In Section
\ref{sect_M-basis}, we consider a version of Theorem
\ref{thm_lim_inf}, namely Proposition \ref{prop_M-basis}, in the more
general context of Markushevich bases, and present the proofs.

We finish this section by making some observations about condition
(\ref{eqn_Sch_est}) above. Let us recall that $B \subseteq S_{X^*}$ is
called a \emph{boundary} of $X$ (with respect to the norm $\ndot$) if,
given $x \in X$, there exists $f \in B$ such that $f(x)=\n{x}$. In
\cite{f:89} and \cite{h:95}, it was proved that every Banach space
that has a $\sigma$-compact boundary (with respect to the norm
topology) admits an equivalent polyhedral norm. We show that, in this
case, condition (\ref{eqn_Sch_est}) is necessary, provided that
$(e_j)_{j=1}^\infty$ is shrinking.

\begin{prop}\label{prop_Sch_est_nec} Assume that $X$ has a shrinking
  Schauder basis and a $\sigma$-compact boundary. Then there exists a
  sequence $(a_n)_{n=1}^\infty$ of positive numbers tending to $0$,
  such that (\ref{eqn_Sch_est}) holds.
\end{prop}

We have need of the following fact, which will be used also in
Corollary \ref{cor_inf_lim_inf}.

\begin{fact}\label{fact_sequences} For every $m \in \N$, let
  $(a_{m,n})_{n=1}^\infty$ be a sequence of positive numbers such that
  $\lim_{n\to\infty} a_{m,n}=0$. Then the sequence
  \begin{equation}\label{eqn_sequence}
    a_n \;:=\; \sum_{m=1}^\infty 2^{-m}\frac{a_{m,n}}{1+a_{m,n}},
    \qquad n \in \N,
  \end{equation}
  tends to $0$, and 
\[
a_{m,n} \;\leqslant\; 2^m a_n\max_{k \in \N} (a_{m,k}+1),
\]
for all $m,n \in \N$.
\end{fact}

\begin{proof}[Proof of Proposition \ref{prop_Sch_est_nec}] Let
  $(K_m)_{m=1}^\infty$ be a sequence of norm compact subsets of
  $S_{X^*}$, such that $B:=\bigcup_{m=1}^\infty K_m$ is a
  boundary. Since $(e_j)_{j=1}^\infty$ is shrinking, it is well known
  that $\lim_{n \to \infty} \n{R_n^*f}=0$ for all $f \in X^*$
  \cite[Proposition 1.b.1]{lt:77}. Using the norm compactness of the
  $K_m$, $m \in \N$, we see that
\[
a_{m,n} \;:=\; \sup_{f \in K_m} \n{R^*_n f},
\]
tends to $0$ as $n\to\infty$. Let $x \in X$. As $B$ is a boundary, there exists $m\in\N$
such that $f(x)=\n{x}$ for some $f \in K_m$. Given $n \in \N$, we have
\[
\n{x} \;=\; f(x) \;=\; f(P_n x) + R^*_nf(x) \;\leqslant\; \n{P_n x} +
\n{R^*_n f}\,\n{x} \;\leqslant\; \n{P_n x} + a_{m,n}\n{x},
\]
hence
\[
\frac{\n{x}-\n{P_n x}}{a_{m,n}} \;\leqslant\; \n{x},
\]
for all $n \in \N$. Defining $a_n$ as in (\ref{eqn_sequence}) yields
\[
\frac{\n{x}-\n{P_n x}}{a_n} \;\leqslant\; 2^m \max_{k\in\N}(a_{m,k}+1) \n{x},
\]
for all $n \in \N$.
\end{proof}

The requirement that the basis in Proposition \ref{prop_Sch_est_nec}
be shrinking is necessary for the conclusion to hold.

\begin{example}\label{ex_shrinking_basis} The space $c_0$ with its
  natural norm has a countable boundary, but with respect to the
  summing basis of $c_0$, there is no sequence $(a_n)_{n=1}^\infty$
  tending to $0$, such that (\ref{eqn_Sch_est}) holds.
\end{example}

\begin{proof} Let $(e_j)_{j=1}^\infty$ and $(e_j^*)_{j=1}^\infty$ be
  the standard bases of $c_0$ and $\ell_1$, respectively. The set
  $\set{\pm e_j^*}{j \in \N}$ is a countable boundary of $c_0$ with
  respect to its natural norm. If $x_j :=\sum_{i=1}^j e_i$ denotes the
  $j$th element of the summing basis of $c_0$, then $x^*_j =
  e_j^*-e_{j+1}^*$, and with respect to this basis we see that
\[
P_n x \;=\; \sum_{j=1}^n (x(j)-x(j+1))\left(\sum_{i=1}^j e_i\right)
\;=\;\sum_{i=1}^n \big(x(i)-x(n+1)\big)e_i.
\]
Suppose that $x(1)=\pn{x}{\infty}\geqslant |x(n)|+1$ whenever $n
\geqslant 2$. Then, whenever $|x(n+1)|\leqslant \frac{1}{2}$, we have
$\pn{P_n x}{\infty} = x(1)-x(n+1) = \pn{x}{\infty}-x(n+1)$. Given a
sequence $(a_j)_{j=1}^\infty$ of positive numbers tending to $0$,
define $x \in c_0$ by $x(1)=\max_{j\geqslant 1} a_j^{\frac{1}{2}}+1$
and $x(j)=a_{j-1}^\frac{1}{2}$ for $n\geqslant 2$. Fix $m \in \N$ such
that $a_n \leqslant \frac{1}{4}$ whenever $n\geqslant m$. Then
$|x(n+1)| \leqslant \frac{1}{2}$ for such $n$ and
\[
a_n^{-1}(\pn{x}{\infty}-\pn{P_n x}{\infty}) \;=\; a_n^{-1}x(n+1) \;=\;
a_n^{-\frac{1}{2}} \;\to\; \infty. \qedhere
\]
\end{proof}

\section{Examples}\label{sect_examples}

In our first example, we present two wide classes of Banach spaces
that are quite different in character, yet share the property of
having {\st}.

\begin{example}\label{ex_two_classes}$\;$
\begin{enumerate}
\item Let $X$ have a normalized unconditional basis
  $(e_\gamma)_{\gamma \in \Gamma}$ and suppose that the set of all
  \emph{summable elements} of the unit sphere
\[
\set{f \in S_{X^*}}{\sum_{\gamma \in \Gamma} |f(e_\gamma)| < \infty},
\]
with respect to the basis, is a boundary. Then $X$ has {\st}.
\item\label{ex_Orlicz} Let $M$ be a non-degenerate normalized Orlicz
  function, i.e.~$M(t)>0$ for all $t>0$ and $M(1)=1$. Let $\Gamma$ be
  a set and let $h_M(\Gamma)$ be the space of all real functions $x$
  defined on $\Gamma$, such that
\[
\sum_{\gamma \in \Gamma} M\left(\frac{|x(\gamma)|}{\rho}\right) \;<\; \infty,
\]
for all $\rho > 0$. We equip $h_M(\Gamma)$ with the Luxemburg norm
\[
\n{x} \;:=\; \inf\set{\rho>0}{ M\left(\frac{|x(\gamma)|}{\rho}\right) \leqslant 1}.
\]
The space $h_M(\Gamma)$ has {\st} with respect to the unit vector
basis $(e_\gamma)_{\gamma \in \Gamma}$, provided
\begin{equation}\label{eqn_Leung}
\lim_{t\to 0} \frac{M(Kt)}{M(t)} \;=\; \infty,
\end{equation}
for some constant $K>1$.
\end{enumerate}
\end{example}

\begin{proof}$\;$
\begin{enumerate}
\item Set $\omega(t)=t$. Given $x \in X$, take $f \in B$ such that
  $f(x)=\n{x}$. Set $c(x) = \sum_{\gamma \in \Gamma}|f(e_\gamma)|$ and
  $d(x)=1$. Given $A \subseteq \Gamma$,
\begin{align*}
\n{x} \;=\; f(x) &\;=\; f(P_A x) + f(R_A x)\\
&\;=\; f(P_Ax) + \sum_{\gamma \in \Gamma\setminus A} f(e_\gamma)e^*_\gamma(x)\\
&\;\leqslant\; \n{P_A x} + \bigg(\sum_{\gamma \in \Gamma\setminus A}
  |f(e_\gamma)| \bigg) \pn{R_A x}{\infty}\\
&\;\leqslant\; \n{P_A x} + c(x) \pn{R_A x}{\infty}.
\end{align*}
\item Given $t>0$, set
\[
\omega(t) \;=\; \sup\set{\frac{M(\tau)}{M(K\tau)}}{0 < \tau \leqslant t}. 
\]
Evidently, $\omega$ is a continuous non-decreasing function and
$\lim_{t\to 0} \omega(t)=0$. Given $x = \sum_{\gamma \in \Gamma}
x(\gamma)e_\gamma \in h_M(\Gamma)$, $\n{x}=1$, we let
\[
c(x) \;=\; \sum_{\gamma \in \Gamma} M(K|x(\gamma)|),
\]
and $d(x)=1$. From the definition of $h_M(\Gamma)$, we see that $c(x)$
is finite. Let $A \subseteq \Gamma$. Since $M$ is a convex function
satisfying $M(0)=0$, we have
\[
\sum_{\gamma \in \Gamma} M(\lambda |x(\gamma)|) \;\leqslant\;
\lambda\sum_{\gamma \in \Gamma} M(|x(\gamma)|),
\]
whenever $0 \leqslant \lambda \leqslant 1$. In particular, as $\n{P_A x} \leqslant 1$,
\[
\sum_{\gamma \in A} M(|x(\gamma)|) \;\leqslant\; \n{P_A x}\sum_{\gamma
  \in A} M\bigg(\frac{|x(\gamma)|}{\n{P_A x}} \bigg) \;=\; \n{P_A x}.
\]
Therefore,
\begin{align*}
\n{x} \;=\; 1 &\;=\; \sum_{\gamma \in \Gamma} M(|x(\gamma)|)\\
&\;=\; \sum_{\gamma \in A} M(|x(\gamma)|) + \sum_{\gamma \in
  \Gamma\setminus A} M(|x(\gamma)|)\\
&\;\leqslant\; \n{P_A x} + \bigg( \sup_{\gamma \in \Gamma\setminus A}
  \frac{M(|x(\gamma)|)}{M(K|x(\gamma)|)}\bigg)\sum_{\gamma \in
  \Gamma\setminus A} M(K|x(\gamma)|)\\
&\;\leqslant\; \n{P_A x} + \omega(\pn{R_A x}{\infty})\sum_{\gamma \in
  \Gamma} M(K|x(\gamma)|)\\
&\;=\; \n{P_A x} + c(x)\omega(\pn{R_A x}{\infty}) \tag*{\qedhere}
\end{align*}
\end{enumerate}
\end{proof}

\begin{remark}$\;$
\begin{enumerate}
\item For the use of summable boundaries in polyhedral renorming, see
  \cite{bs:16,fst:14}.
\item D.~Leung proved that $h_M(\N)$ admits an equivalent polyhedral
  norm provided $M$ satisfies (\ref{eqn_Leung}) \cite{l:94}. For the
  case when $\Gamma$ is an arbitrary set, see \cite{fpst:08,fpst:14}.
\end{enumerate}
\end{remark}

\begin{example}\label{ex_Nakano} We consider a symmetric version of
  the Nakano space. Let $\Gamma$ be a set and let $(p_n)_{n=1}^\infty$ be
  a non-decreasing sequence, with $p_1 \geqslant 1$. By
  $h^S_{(p_n)}(\Gamma)$ we denote the space of all real functions $x$
  defined on $\Gamma$, such that
\[
\phi\bigg(\frac{x}{\rho}\bigg) \;<\; \infty,
\]
for all $\rho > 0$, where
\[
\phi(x) \;:=\; \sup\set{\sum_{k=1}^\infty
  |x(\gamma_k)|^{p_k}}{(\gamma_k)_{k=1}^\infty \text{ is a sequence of
    distinct points in }\Gamma}.
\]
Given $x \in h^S_{(p_n)}(\Gamma)$, we set
\[
\n{x} \;=\; \inf\set{\rho>0}{\phi\bigg(\frac{x}{\rho}\bigg) \leqslant 1}.
\]
It is easy to see that the standard unit vectors $(e_\gamma)_{\gamma
  \in \Gamma}$ form an unconditional symmetric basis in
$h^S_{(p_n)}(\Gamma)$. We show that $h^S_{(p_n)}(\Gamma)$ satisfies
equation (\ref{eqn_unc_est}) from Theorem \ref{thm_lim_inf}, provided
$p_n \to \infty$.
\end{example}

\begin{proof}
Pick $\theta \in (0,1)$. We show that for every $x \in
h^S_{(p_n)}(\Gamma)$ satisfying $\n{x}=1$, there exists $m(x) \in \N$
such that
\begin{equation}\label{eqn_Nakano}
1-\n{P_{A_n(x)} x} \;\leqslant\; \theta^{p_n},
\end{equation}
whenever $n \geqslant m(x)$, where $A_n(x)$ is any set
$\{\gamma_1,\dots,\gamma_n\}$ of the form described in Remark
\ref{rem_A_n(x)}. Setting $a_n=\theta^{p_n}$ in (\ref{eqn_Nakano})
yields (\ref{eqn_unc_est}).

As in the proof of Example \ref{ex_two_classes} (\ref{ex_Orlicz}), as
$\phi$ is a convex function and $\phi(0)=0$, and $\n{P_{A_n(x)}x}
\leqslant \n{x} = 1$, we have
\begin{equation}\label{eqn_Nakano_1}
\phi(P_{A_n(x)}x) \;\leqslant\;
\n{P_{A_n(x)}x}\phi\left(\frac{P_{A_n(x)}x}{\n{P_{A_n(x)}}}\right)
\;=\; \n{P_{A_n(x)}x}.
\end{equation}
Given $\gamma \in \Gamma\setminus A_n(x)$, and bearing in mind that
$\ndot$ is a lattice norm, we have
\[
|x(\gamma)| \;\leqslant\; \n{R_{A_n(x)} x} \;\leqslant\; \n{x} \;=\; 1,
\]
and therefore
\begin{align*}
  1 \;=\; \phi\left(\frac{R_{A_n(x)}x}{\|R_{A_n(x)}x\|}\right) \;
&=\; \sum_{k=1}^\infty\left(\frac{|x(\gamma_{n+k})|}{\|R_{A_n(x)}x\|}\right)^{p_k}\\
&\geqslant\;
  \sum_{k=1}^\infty\left(\frac{|x(\gamma_{n+k})|}{\|R_{A_n(x)}x\|}\right)^{p_{k+n-1}}\\
&=\;
  \sum_{j=n+1}^\infty\left(\frac{|x(\gamma_j)|}{\|R_{A_n(x)}x\|}\right)^{p_{j-1}}
  \;\geqslant\;
  \sum_{j=n+1}^\infty\frac{|x(\gamma_j)|^{p_j}}{\n{R_{A_n(x)}x}^{p_n}},
\end{align*}
which implies
\begin{align}\label{eqn_Nakano_2}
\sum_{j=n+1}^\infty |x(\gamma_j)|^{p_j} \;\leqslant\; \n{R_{A_n(x)}x}^{p_n}.
\end{align}
There exists $m(x) \in \N$ such that $\n{R_{A_n(x)}x} \leqslant
\theta$ whenever $n \geqslant m(x)$. Together with
(\ref{eqn_Nakano_1}) and (\ref{eqn_Nakano_2}), this implies
\begin{align*}
1 \;=\; \phi(x) \;&=\; \phi(P_{A_n(x)}x) + \sum_{j=n+1}^\infty |x(\gamma_j)|^{p_j}\\
&\leqslant\; \n{P_{A_n(x)}x} + \n{R_{A_n(x)}x}^{p_n} \;\leqslant\;
  \n{P_{A_n(x)}x} + \theta^{p_n},
\end{align*}
whenever $n \geqslant m(x)$.
\end{proof}

The following examples are based on the next simple and well known fact.

\begin{fact} Let $(c_k)_{k=1}^n$ and $(d_k)_{k=1}^n$ be non-increasing
  sequences of non-negative numbers. Then
\begin{equation}\label{eqn_perm}
\sum_{k=1}^n c_k d_{\pi(k)} \;\leqslant\; \sum_{k=1}^n c_k d_k,
\end{equation}
whenever $\pi$ is a permutation of $\{1,\dots,n\}$.
\end{fact}

In the next example, we expose the difference between conditions
(\ref{eqn_unc_est}) and (\ref{eqn_Sch_est}) of Theorem
\ref{thm_lim_inf}.

\begin{example} There exists an equivalent norm $\ndot$ on $c_0$ that
  is symmetric with respect to the usual basis, such that
\begin{enumerate}
\item given $x \in c_0$,
\begin{equation}\label{eqn_Day_1}
2^n\bigg(\n{x}-\sup_{|A|\leqslant n} \n{P_A x} \bigg) \;\leqslant\; 4\n{x}, 
\end{equation}
\item but given a sequence $(a_n)_{n=1}^\infty$ of positive numbers
  tending to $0$, there exists $y \in c_0$ such that
\begin{equation}\label{eqn_Day_2}
\lim_{n\to\infty} a_n^{-1}\big(\n{y}-\n{P_n y} \big) \;=\; \infty.
\end{equation}
\end{enumerate} 
\end{example}

\begin{proof}
Consider Day's norm, defined on $c_0$ by
\begin{equation}\label{eqn_Day_norm}
\n{x} \;=\; \sup \set{\bigg(\sum_{k=1}^\infty
  2^{-k}x(j_k)^2\bigg)^{\frac{1}{2}}}{(j_k)_{k=1}^\infty \text{ is a
    sequence of distinct points in }\N}.
\end{equation}
\begin{enumerate}
\item Pick $x \in c_0$ such that $\n{x}=1$. We define $A_n(x)$ as
in Remark \ref{rem_A_n(x)}. From
  (\ref{eqn_perm}), it follows that
\begin{equation}\label{eqn_Day_3}
\sup_{|A|\leqslant n} \n{P_A x} \;=\; \n{P_{A_n(x)}x} \;=\;
\bigg(\sum_{k=1}^n 2^{-k}x(j_k)^2 \bigg)^{\frac{1}{2}}.
\end{equation}
Since $|x(\gamma)| \leqslant 2\n{x}=2$, we have
\[
1-\n{P_{A_n(x)}x} \;\leqslant\; 1-\n{P_{A_n(x)}}^2 \;=\;
\sum_{k=n+1}^\infty 2^{-k}x(j_k)^2 \;\leqslant\; 4\sum_{k=n+1}^\infty
2^{-k} \;=\; 2^{2-n}.
\]
Together with (\ref{eqn_Day_3}), this implies (\ref{eqn_Day_1}).
\item Let $(a_n)_{n=1}^\infty$ be a sequence of positive numbers
  tending to $0$. Let $(n_k)_{k=1}^\infty$ be a strictly increasing
  sequence of positive integers such that
\begin{equation}\label{eqn_Day_4}
a_n \;\leqslant\; 8^{-k},
\end{equation}
for all $n \geqslant n_k$. Define $x \in c_0$ by
\[
x(n) \;=\; \begin{cases} 3^{\frac{1}{2}}\cdot 2^{-\frac{k}{2}} &
  \text{if }n=n_k, \\ 0 & \text{otherwise.} \end{cases}
\]
From (\ref{eqn_Day_norm}) we get $\n{x}=1$ and 
\[
1 - \n{P_n x}^2 \;=\; 3\sum_{i=k+1}^\infty 4^{-i} \;=\; 4^{-k},
\]
whenever $n_k \leqslant n < n_{k+1}$. Hence,
\[
1-\n{P_n x} \;>\; {\ts\frac{1}{2}}(1-\n{P_n x}^2) \;=\; {\ts\frac{1}{2}}4^{-k}.
\]
Using (\ref{eqn_Day_4}), we obtain $a_n^{-1}(1-\n{P_nx}) \geqslant
2^{k-1}$ whenever $n \geqslant n_k$, which yields
(\ref{eqn_Day_2}). \qedhere
\end{enumerate}
\end{proof}

The next example shows that condition (\ref{eqn_Sch_est}) of Theorem
\ref{thm_lim_inf} can fail even on $c_0$, if the norm fails to be
symmetric.

\begin{example}\label{ex_non-sym} There exists on $c_0$ an equivalent
  (non-symmetric) norm $\ndot$, with respect to which the standard
  basis is normalized and $1$-unconditional, and having the property
  that given a sequence $(a_n)_{n=1}^\infty$ of positive numbers
  tending to $0$, there exists $x \in c_0$ such that
\begin{equation}\label{eqn_non-sym}
\lim_{n\to\infty} a_n^{-1}\bigg(\n{x}-\sup_{|A|\leqslant n} \n{P_A x} \bigg) \;=\; \infty.
\end{equation}
\end{example}

\begin{proof} Let 
\[
D\;=\; \set{m^{-1}2^{-k}}{m,k \in \N},
\]
and let $\map{q}{\N}{D}$ have the property that $q^{-1}(d)$ is
infinite for all $d \in D$. Write $q_n=q(n)$, $n \in \N$. Let $S$ be
the set of all infinite subsets $L \subseteq \N$, such that $q_j
\geqslant q_n$ whenever $j,n \in L$, $j \leqslant n$, and $\sum_{n \in
  L} q_n=1$. Set
\[
E \;=\; \set{\pm e_n^*}{n \in \N} \cup \set{2\sum_{n \in L}s_n q_n
  e_n^*}{L \in S\text{ and } s_n \in \{-1,1\}\text{ for all }n\in\N},
\]
and define the norm
\[
\n{x} \;=\;\set{f(x)}{f \in E}.
\]
Then $\pn{x}{\infty} \leqslant \n{x} \leqslant 2\pn{x}{\infty}$ and
$\n{e_n}=1$, as $q_n \leqslant \frac{1}{2}$ for all $n$, and the signs
$s_n$ in the definition of $E$ ensure that the standard basis is
$1$-unconditional with respect to $\ndot$.

Given $x \in c_0$, we shall say that $|x|$ is non-increasing on its
support if $|x(j)| \geqslant |x(n)|$ whenever $j,n \in \supp(x)$ and
$j \leqslant n$. Next, we prove the following fact. Let $x \in c_0$
such that $|x|$ is non-increasing on its support, and suppose that
there exists $L \in S$ such that $\supp(x)\subseteq L$ and
\[
\pn{x}{\infty} \;<\; 2\sum_{j \in L} q_j|x(j)|.
\]
Furthermore, let $L_n$ be the set of the first $n$ elements of $L$,
and let $n_0$ be large enough so that 
\[
\pn{x}{\infty} \;<\; 2\sum_{j \in L_{n_0}} q_j|x(j)|.
\]
Then the conclusion is that
\begin{equation}\label{eqn_non-sym_1}
\n{x}-\sup_{|A| \leqslant n} \n{P_A x} \;=\; 2\sum_{j \in L\setminus L_n} q_j|x(j)|,
\end{equation}
whenever $n \geqslant n_0$.

To prove this fact, first we show that
\begin{equation}\label{eqn_norm_attainment}
2\sum_{j \in L} q_j|x(j)| \;=\; \n{x}.
\end{equation}
One inequality is obvious. To see the other, since $\pn{x}{\infty} <
2\sum_{j \in L} q_j|x(j)|$, all we need to do is check that
\[
\sum_{j \in M} q_j|x(j)| \;\leqslant\; \sum_{j \in L} q_j|x(j)|,
\]
whenever $M \in S$, and indeed this holds, because $\supp(x)\subseteq
L$. Next, since $|x|$ is non-increasing on its support, as is
$(q_j)_{j \in L}$, given $n \geqslant n_0$ and $A \subseteq \N$,
$|A|\leqslant n$, we have
\[
\n{P_A x} \;\leqslant\; 2\sum_{j \in L_n} q_j|x(j)| \;=\; \n{P_{L_n}x}.
\]
The equality in the line above follows because
(\ref{eqn_norm_attainment}) holds with $P_{L_n}x$ and $L_n$ in place
of $x$ and $L$, respectively. Note that
\[
\pn{P_{L_n} x}{\infty} \;<\; 2\sum_{j \in L_n} q_j|x(j)|,
\]
whenever $n \geqslant n_0$. Since $|L_n|=n$, this completes the proof of the fact.

Now let $(a_n)_{n=1}^\infty$ be a sequence of positive numbers tending
to $0$. Choose integers $0=n_0 < n_1 < n_2 < \dots$, such that
\begin{enumerate}
\item[(a)] $a_n \leqslant 8^{-k}$ whenever $n \geqslant n_k$, and 
\item[(b)] $n_k-n_{k-1} \leqslant n_{k+1}-n_k$ for all $k \in \N$.
\end{enumerate}
Since $q^{-1}(d)$ is infinite for all $d \in D$, it is possible to
find finite sets $H_k \subseteq \N$ such that
\begin{enumerate}
\item[(c)] $\max H_k < \min H_{k+1}$,
\item[(d)] $|H_k|=n_k-n_{k-1}$ and
\item[(e)] $q_j =2^{-k}/|H_k|$ for all $j \in H_k$.
\end{enumerate}
Define $L=\bigcup_{k=1}^\infty H_k$. We have
\[
\sum_{j \in L} q_j \;=\; \sum_{k=1}^\infty \sum_{j \in H_k}
\frac{2^{-k}}{|H_k|} \;=\; \sum_{k=1}^\infty 2^{-k} \;=\; 1.
\]
Together with (b) -- (e) above, this ensures that $L \in S$. Now define $x \in c_0$ by
\[
x(j) \;=\; \begin{cases}\frac{3}{2}\cdot 2^{-k} & \text{whenever }j
  \in H_k, \\ 0 & \text{otherwise.}
\end{cases}
\]
Then $|x|=x$ is non-increasing on its support, which equals $L$, and
\[
2\sum_{j \in L} q_j |x(j)| \;=\; 2\sum_{k=1}^\infty \sum_{j \in K_k}
\frac{2^{-k}}{|H_k|}\cdot {\ts\frac{3}{2}}\cdot 2^{-k} \;=\;
3\sum_{k=1}^\infty 4^{-k} \;=\; 1 \;>\; {\ts \frac{3}{4}} \;=\; |x(1)|
\;=\; \pn{x}{\infty}.
\]
We make the simple observation that
\[
2\sum_{j \in L_{n_2}} q_j|x(j)| \;=\; 3\left(\sum_{j \in H_1}
  \frac{4^{-1}}{|H_1|} + \sum_{j \in H_2} \frac{4^{-2}}{|H_2|} \right)
\;=\; 3({\ts \frac{1}{4} + \frac{1}{16}}) \;>\; \pn{x}{\infty}.
\]
Therefore, using equation (\ref{eqn_non-sym_1}), given $n \geqslant n_2$, we have
\[
\n{x}-\sup_{|A| \leqslant n} \n{P_A x} \;=\; 2\sum_{j \in L\setminus L_n} q_j x(j).
\]
Given $n \geqslant n_2$, let $k \geqslant 2$ such that $n_k \leqslant
n < n_{k+1}$. Then
\begin{align*}
\n{x}-\sup_{|A| \leqslant n} \n{P_A x} \;=\; 2\sum_{j \in L\setminus
  L_n} q_j x(j)\;&\geqslant\; 2\sum_{j \in  L\setminus L_{n_{k+1}}}
                   q_j x(j)\\
&=\; 2\sum_{\ell=k+2}^\infty \sum_{j \in H_\ell} q_j x(j) \;=\;
  2\sum_{\ell=k+2}^\infty {\ts \frac{3}{2}\cdot 4^{-\ell}} \;=\;
  4^{-k-2}.
\end{align*}
Combining this with (a) above yields
\[
a_n^{-1}\bigg(\n{x}-\sup_{|A|\leqslant n} \n{P_A x} \bigg)
\;\geqslant\; 8^k 4^{-k-2} \;=\; 4^{k-2} \;\to\; \infty,
\]
as $n \to \infty$.
\end{proof}

We do not know if the norm in Example \ref{ex_non-sym} can be replaced
by one that is symmetric.

\begin{prob} Let $X=(c_0,\ndot)$, where $\ndot$ is a symmetric
  equivalent norm. Does there exist a sequence $(a_n)_{n=1}^\infty$ of
  positive numbers tending to $0$, such that (\ref{eqn_unc_est}) holds
  for all $x \in X$?
\end{prob}

\section{Decompositions of Banach spaces having a Markushevich basis}\label{sect_M-basis}

Let $(e_\gamma,e^*_\gamma)_{\gamma \in \Gamma}$ be a strong normalized
{\em Markushevich basis} (M-basis for short),
i.e.~$e^*_\beta(e_\gamma)=\delta_{\beta \gamma}$ for all $\beta,\gamma
\in \Gamma$, $\n{e_\gamma}=1$ for all $\gamma \in \Gamma$,
$\cl{\lspan}^{w*}(e^*_\gamma)=X^*$ and
\[
x \in \cl{\lspan}^{\ndot}\set{e_\gamma}{\gamma \in \supp(x)}.
\]
The next result is the main tool we use to prove Theorem \ref{thm_lim_inf}.

\begin{prop}[{\cite[Corollary 14]{fpst:14}}]\label{prop_fpst}
Let $X$ have a strong M-basis $(e_\gamma,e^*_\gamma)_{\gamma \in
  \Gamma}$ and suppose that we can write
\[
S_X = \bigcup_{k=1}^\infty S_k,
\]
and find a sequence of positive integers $n_k$ in such a way that the sequence
\[
b_k \;:=\; \inf_{x \in S_k}\sup\set{f(x)}{f \in S_{X^*} \text{ and }|\supp(f)| \leqslant n_k},
\]
is strictly positive and converges to $1$. Then $X$ admits a
polyhedral renorming. Moreover, if $\Gamma=\N$ and the sequence
\[
c_k \;:=\; \inf_{x \in S_k}\sup\set{f(x)}{f \in S_{X^*} \text{ and
  }\max(\supp(f)) \leqslant n_k},
\]
behaves likewise, then we reach the same conclusion.
\end{prop}

\begin{prob} Let $X=(c_0,\ndot)$ be as in Example
  \ref{ex_non-sym}. Does $S_X$ satisfy the hypothesis of Proposition
  \ref{prop_fpst}?
\end{prob}

\begin{prop}\label{prop_M-basis} Let $X$ be a Banach space with a
  strong M-basis $(e_\gamma,e^*_\gamma)_{\gamma \in \Gamma}$. Let
  $(a_n)_{n=1}^\infty$ be a sequence of positive numbers tending to
  $0$, such that
\begin{equation}\label{eqn_M-basis_1}
\lim\inf_{n\to\infty} a_n^{-1}\big(\n{x}-\sup\set{f(x)}{f \in S_{X^*}
  \text{ and }|\supp(f)|\leqslant n} \big) \;<\; \infty,
\end{equation}
for all $x \in X$. Then $X$ admits an equivalent polyhedral norm. If
$\Gamma=\N$ and
\begin{equation}\label{eqn_M-basis_2}
\lim\inf_{n\to\infty} a_n^{-1}\big(\n{x}-\sup\set{f(x)}{f \in S_{X^*}
  \text{ and }\max(\supp(f))\leqslant n} \big) \;<\; \infty,
\end{equation}
then we reach the same conclusion.
\end{prop}

\begin{proof} We consider the first case. Without loss of generality,
  we may assume that the sequence $(a_n)_{n=1}^\infty$ is
  non-increasing (if necessary, we can replace $a_n$ by $a'_n:=\max_{j
    \geqslant n} a_j$ -- clearly (\ref{eqn_M-basis_1}) holds with
  respect to the $a'_n$). There exists an increasing sequence of
  positive integers $(n_k)_{k=1}^\infty$, such that the sequence
  $(ka_{n_k})_{k=1}^\infty$ tends to $0$ and $\max_{k \geqslant 1}
  ka_{n_k} < 1$. From (\ref{eqn_M-basis_1}) it follows that, for every
  $x \in S_X$, there exist positive integers $m(x)$ and $\ell(x)
  \geqslant n_{m(x)}$ such that
\begin{equation}\label{eqn_M-basis_3}
1 \;\leqslant\; \sup\set{f(x)}{f \in S_{X^*} \text{ and
  }|\supp(f)|\leqslant n_{\ell(x)}} + m(x)a_{\ell(x)}.
\end{equation}
Given $k \in \N$, set
\[
S_k \;=\; \set{x \in S_X}{n_k \leqslant \ell(x) < n_{k+1}},
\]
and let $K=\set{k \in \N}{S_k \text{ is non-empty}}$. Clearly, $S_X
\;=\; \bigcup_{k \in K} S_k$.

Let $k \in K$ and $x \in S_k$. We have $\max\{n_k,n_{m(x)}\} \leqslant
\ell(x) < n_{k+1}$. Since $(n_k)_{k=1}^\infty$ is increasing and
$(a_k)_{k=1}^\infty$ is non-increasing, we get $m(x)\leqslant k$ and
$a_{\ell(x)} \leqslant a_{n_k}$. Using (\ref{eqn_M-basis_3}) we get
\[
1 \;\leqslant\; \sup\set{f(x)}{f \in S_{X^*} \text{ and
  }|\supp(f)|\leqslant n_k} + ka_{n_k}.
\]
Given $k \in K$, define
\[
b_k \;=\; \inf_{x \in S_k}\sup\set{f(x)}{f \in S_{X^*} \text{ and }|\supp(f)|\leqslant n_k}.
\]
We obtain $0<1-ka_{n_k} \leqslant b_k \leqslant 1$. Enumerate $K$ as
an increasing sequence of positive integers
$(k_j)_{j=1}^\infty$. Clearly $(S_{k_j})_{j=1}^\infty$ and
$(b_{k_j})_{j=1}^\infty$ satisfy the hypothesis of Proposition
\ref{prop_M-basis}.

When $\Gamma=\N$, we repeat the proof above, using
(\ref{eqn_M-basis_2}), replacing $|\supp(f)|$ by $\max(\supp(f))$ as
we go, and using the second part of Proposition \ref{prop_fpst}.
\end{proof}

\begin{proof}[Proof of Theorem \ref{thm_lim_inf}] First let us show
  how to treat the case when the basis $(e_\gamma)_{\gamma \in
    \Gamma}$ is unconditional. Given $x \in X$ and $\alpha \in
  [-1,1]^\Gamma$, we set
\[
\psi(x,\alpha) \;=\; \n{\sum_{\gamma \in \Gamma}
  \alpha(\gamma)e_\gamma^*(x)e_\gamma}.
\]
Introduce on $X$ an equivalent norm by the formula
\[
\tn{x} \;=\; \sup\set{\psi(x,\alpha)}{\alpha \in [-1,1]^\Gamma}.
\]
Let (\ref{eqn_unc_est}) hold. We show that, for every $x \in X$,
(\ref{eqn_M-basis_1}) holds with respect to $\tndot$.

Let $x \in X$. Since the function $\psi$ is a continuous with respect
to its second argument, and as $[-1,1]^\Gamma$ is compact, we find
that $\psi(x,\cdot)$ attains its maximum at some $\beta \in
[-1,1]^\Gamma$, i.e.~$\tn{x}=\psi(x,\beta)$. Set $y=\sum_{\gamma \in
  \Gamma}\beta(\gamma)e_\gamma^*(x) e_\gamma$.  From the definition of
$\tndot$, we know that
\[
\n{P_A y} \;\leqslant\; \tn{P_A y} \;=\; \tn{P_A x},
\]
for every $A \subseteq \Gamma$. Hence
\[
\sup_{|A|\leqslant n} \n{P_A y} \;\leqslant\; \sup_{|A|\leqslant n} \tn{P_A x}.
\]
Since $\tn{x}=\n{y}$, we get
\[
\tn{x}-\sup_{|A|\leqslant n} \tn{P_A x} \;\leqslant\; \n{y} - \sup_{|A|\leqslant n} \n{P_A y}.
\]
Bearing in mind that (\ref{eqn_unc_est}) holds for $y$, we have
\begin{equation}\label{eqn_unc_est_1}
\lim\inf_{n\to\infty} a_n^{-1}\big(\tn{x}-\sup_{|A|\leqslant n} \tn{P_A x} \big) \;<\; \infty.
\end{equation}
Since the basis is unconditionally monotone with respect to $\tndot$, we obtain
\[
\tn{P_A x} \;=\; \sup\set{f(x)}{\tn{f}=1 \text{ and }\supp(f)=A}.
\]
This, together with (\ref{eqn_unc_est_1}), shows that
(\ref{eqn_M-basis_1}) holds with respect to $\tndot$. Thus we can
apply Proposition \ref{prop_M-basis}.

In the Schauder basis case, we proceed much as above, using the
equivalent norm $\tn{x}=\sup_n\n{P_n x}$. First, we show that
(\ref{eqn_M-basis_2}) holds with respect to $\tndot$. Let $x \in
X$. If $\tn{x}$ should happen to equal $\n{x}$, we have
\begin{equation}\label{eqn_unc_est_2}
\tn{x}-\tn{P_n x} \;\leqslant\; \n{x}-\n{P_n x},
\end{equation}
for all $n \in \N$. Assume now that $\tn{x}>\n{x}$. Since
$\n{x}=\lim_{n \to \infty} \n{P_n x}$, we have $\tn{x}=\n{P_m x}$ for
some $m \in \N$. Since the basis $(e_j)_{j=1}^\infty$ is monotone with
respect to $\tndot$ we have $\tn{P_n x}=\tn{P_m x}$ whenever $n
\geqslant m$. Thus in this case
\[
\tn{x}-\tn{P_n x} \;=\; 0.
\]
Together with (\ref{eqn_unc_est_2}), this implies that
(\ref{eqn_Sch_est}) holds with respect to $\tndot$. Again, given that
the basis is monotone with respect to $\tndot$, we find that
\[
\tn{P_n x} \;=\; \sup\set{f(x)}{f \in X^*,\;\tn{f}\leqslant 1 \text{
    and }\max(\supp(f)) \leqslant n}.
\]
So (\ref{eqn_M-basis_2}) holds with respect to $\tndot$ and we are in
a position to apply Proposition \ref{prop_M-basis} once more.
\end{proof}

Using Fact \ref{fact_sequences}, the hypotheses of Theorem 1.3 can be relaxed a little. 

\begin{cor}\label{cor_inf_lim_inf} Let $X$ be a Banach space having an
  unconditional basis $(e_\gamma)_{\gamma \in \Gamma}$. Given $m \in
  \N$, let $(a_{m,n})_{n=1}^\infty$ be a sequence of positive numbers
  such that $\lim_{n\to\infty} a_{m,n}=0$ and, for every $x \in X$,
\[
\inf_{m \in \N} \left( \lim \inf_{n \to \infty}
  a_{m,n}^{-1}\left(\n{x}-\sup_{|A| \leqslant n} \n{P_A x}
  \right)\right) \;<\; \infty.
\]
Then $X$ admits a polyhedral renorming.
\end{cor}

At last, we are ready to present the proof of Theorem
\ref{thm_symmetric}. Let $(e_\gamma)_{\gamma \in \Gamma}$ be a
normalized unconditional basis of a Banach space $X$. Set
\[
\lambda_n \;=\; \inf\set{\n{\sum_{\gamma \in A} e_\gamma}}{A \subseteq
  \Gamma,\; |A| \geqslant n}.
\]
Assume that $(e_\gamma)_{\gamma \in \Gamma}$ is a symmetric
basis. Then $X$ is isomorphic to $c_0(\Gamma)$ if and only if the
sequence $(\lambda_n)_{n=1}^\infty$ is bounded (this follows
immediately from the fact that, given a normalized basis
$(e_\gamma)_{\gamma \in \Gamma}$ of a Banach space having
unconditional basis constant $K$, we have
\[
K^{-1}\max_{\gamma \in A} |a_\gamma| \;\leqslant\; \n{\sum_{\gamma \in
    A} a_\gamma e_\gamma} \;\leqslant\; K\max_{\gamma \in A}
|a_\gamma|\n{\sum_{\gamma \in A} e_\gamma}.
\]
for every finite set $A \subseteq \Gamma$ and reals $a_\gamma$,
$\gamma \in A$). Since $c_0(\Gamma)$ is polyhedral, Theorem
\ref{thm_symmetric} follows immediately from the final result of the
paper.

\begin{prop}\label{prop_lim_lambda_infty} Let $X$ have {\st} with
  respect to an unconditional basis $(e_\gamma)_{\gamma \in \Gamma}$
  and some modulus $\omega$. If 
\begin{equation}\label{eqn_lim_lambda_infty}
\lim_{n\to\infty} \lambda_n \;=\; \infty,
\end{equation}
then $X$ admits a polyhedral renorming.
\end{prop}

\begin{proof} Pick $x \in X$ and define the sets $A_n(x)$ as in Remark
\ref{rem_A_n(x)}. Given $n \in \N$, we have
\begin{align*}
\pn{R_{A_n(x)}x}{\infty} \;&=\; \sup\set{|e_\gamma^*(x)|}{\gamma \in
                             \Gamma\setminus A_n(x)}\\
&\leqslant\; |e_{\gamma_n}^*(x)|\\
&\leqslant\; \lambda_n^{-1}\n{\sum_{\gamma \in A_n(x)}
  e_\gamma}\cdot|e_{\gamma_n}^*(x)|\\
&=\; \lambda_n^{-1}\n{\sum_{\gamma \in A_n(x)} e_{\gamma_n}^*(x) e_\gamma}\\
&\leqslant\; \lambda_n^{-1}K\n{\sum_{k\geqslant 1} e_{\gamma_k}^*(x)
  e_{\gamma_k}} \;=\; K\n{x}\lambda_n^{-1},
\end{align*}
where $K$ is the unconditional basis constant of $(e_\gamma)_{\gamma
  \in \Gamma}$. Since $X$ is assumed to have {\st}, it follows that
\begin{align}
\n{x} \;&\leqslant\; \n{P_{A_n(x)} x} +
          c(x)\omega(d(x)\pn{R_{A_n(x)}x}{\infty})\nonumber\\
&\leqslant\; \n{P_{A_n(x)} x} + c(x)\omega(Kd(x)\n{x}\lambda_n^{-1}) \nonumber\\
&\leqslant\; \sup_{|A| \leqslant n} \n{P_A x} +
  c(x)\omega(Kd(x)\n{x}\lambda_n^{-1}).\label{eqn_lim_lambda_infty_1}
\end{align}
Set $a_{m,n}=m\omega(m\lambda_n^{-1})$. Given
(\ref{eqn_lim_lambda_infty}), we see that $\lim_{n\to\infty}
a_{m,n}=0$ for all $m \in \N$. From (\ref{eqn_lim_lambda_infty_1}), it
follows that
\[
\n{x} \;\leqslant\; \sup_{|A|\leqslant n} \n{P_A x} + a_{m,n},
\]
for all $n \in \N$, provided $m \geqslant
\max\{c(x),Kd(x)\n{x}\}$. Now we are in a position to apply Corollary
\ref{cor_inf_lim_inf}. The proof is complete.
\end{proof}

\newpage
\end{document}